\documentclass[portuges,12pt,letter]{article}
\usepackage[centertags]{amsmath}
\usepackage{amsfonts}
\usepackage{newlfont}
\usepackage{amscd}
\usepackage{graphics}
\usepackage{epsfig}
\usepackage{indentfirst}
\usepackage{amsxtra}
\usepackage[latin1]{inputenc}
\usepackage{amssymb, amsmath}
\usepackage{amsthm}
\usepackage[mathscr]{eucal}

\newtheorem{thm}{Theorem}[section]

\title{\bf On central fields in the calculus of variations }
\author{ Fabio Silva Botelho \\ Department of Mathematics \\ Federal University of Santa Catarina - UFSC \\
Florian\'{o}polis, SC - Brazil}
\begin{document}

\maketitle
\abstract{ This article develops sufficient conditions of local optimality for the scalar and
vectorial cases of the calculus of variations. The results are established through the construction of stationary fields which keep invariant what we define as
the  generalized Hilbert  integral. }

\section{Introduction} In this short communication we develop sufficient conditions of local optimality for a relatively large class of problems  in the calculus of variations. The concerning approach  is developed through a generalization of some theoretical results about central fields presented in
\cite{12a}. We address both the scalar and vectorial cases for a domain  in $\mathbb{R}^n$. Finally the Weierstrass Excess function has a fundamental
role in the formal proofs of the main results.

\section{Central fields for the scalar case in the calculus of variations}

Let $\Omega \subset \mathbb{R}^n$ be an open, bounded, simply connected set with a regular (Lipschitzian) boundary denoted by $\partial \Omega$.

Let $f \in C^1(\overline{\Omega} \times \mathbb{R} \times \mathbb{R}^n)$ ($f(x,y,\mathbf{z})$) and $V=C^1(\overline{\Omega}).$
We suppose $f$ is convex in $\mathbf{z}$, which we denote by $f(\underline{x},\underline{y},\mathbf{z})$  be convex.

 Choose $\tilde{x}_1,\ldots, \tilde{x}_n,\tilde{y}_1,\ldots,\tilde{y}_n \in \mathbb{R}$  such that if $x=(x_1,\ldots,x_n) \in \Omega$, then $$\tilde{x}_k<x_k<\tilde{y}_k \; \forall k \in \{1, \ldots, n\}$$

Suppose that $y_0 \in V_1=C^1(B_0)$ is stationary for $f$ in $B_0=\prod_{k=1}^n [\tilde{x}_k, \tilde{y}_k]$, that is, suppose
$$\sum_{k=1}^n\frac{d}{dx_k} f_{z_k}(x,y_0(x),\nabla y_0(x))=f_y(x,y_0(x),\nabla y_0(x)),\; \forall x \in B_0.$$
We also assume $f$ to be of $C^1$ class in $B_0 \times \mathbb{R} \times \mathbb{R}^n$.

Consider the problem of minimizing  $F:D \rightarrow \mathbb{R}$ where
$$F(y)=\int_\Omega f(x,y(x), \nabla y(x))\;dx,$$ and where
$$D=\{ y \in V\;:\; y=y_0 \text{ in } \partial \Omega\}.$$

Assume there exists a family of functions $\mathcal{F}$, such that for each $(x,y) \in D_1 \subset \mathbb{R}^{n+1}$,
there exists a unique stationary function $y_3 \in V_1=C^1(B_0)$ for $f$ in $\mathcal{F}$, such that
$$(x,y)=(x,y_3(x)).$$

More specifically, we define $\mathcal{F}$ as a subset of $\mathcal{F}_1$, where
$$\mathcal{F}_1=\{\phi(t,\Lambda(x,y)) \text{ stationary for } f \text{ such that } \phi(x,\Lambda(x,y))=y \text{ for a }
\Lambda \in B_r(0) \subset \mathbb{R}\}$$ where $B_r(0)=(0-r,0+r)$ for some $r>0.$

Here, $\phi$ is stationary and $$\phi((t_1,\ldots,t_k=\tilde{x}_k,\ldots,t_n),\Lambda(x,y))=y_0(t_1,\ldots,t_k=\tilde{x}_k,\ldots,t_n), \text{ on } \partial B_0$$
and $$\nabla\phi((t_1,\ldots,t_k=\tilde{x}_k,\ldots,t_n),\Lambda(x,y))\cdot \mathbf{n}=\Lambda \in \mathbb{R}, \text{ on } \partial B_0,\; \forall k \in \{1,\ldots,n\},$$ where $\mathbf{n}$ denotes the  outward normal field to $\partial B_0$.

Define the field $\theta :D_1 \rightarrow \mathbb{R}^{n}$ by $$\theta(x,y)=\nabla y_3(x),$$ where as above indicated $y_3$ is such that
$$(x,y)=(x,y_3(x)),$$ so that $$\theta(x,y_3(x))=\nabla y_3(x), \; \forall x \in B_0.$$

At this point we assume the hypotheses of the  implicit function theorem so that $\phi(x,\Lambda(x,y))$ is of  $C^1$ class and $\theta(x,y)$ is continuous
(in fact, since $\phi$ is stationary, the partial derivatives of $\theta(x,y)$ are well defined).

Define also $$h(x,y)=f(x,y,\theta(x,y))-\sum_{j=1}^nf_{z_j}(x,y,\theta(x,y)) \theta_j(x,y)$$
and $$P_j(x,y)=f_{z_j}(x,y,\theta(x,y)),\; \forall j \in \{1, \ldots,n\}.$$

Observe that \begin{eqnarray}h_y(x,y_3(x))&=&f_y(x,y_3(x),\theta(x,y_3(x)))+\sum_{j=1}^n f_{z_j}(x,y_3(x),\theta(x,y_3(x))) (\theta_j)_y(x,y_3(x))
\nonumber \\ &&- \sum_{j=1}^n \{(P_j)_y(x,y_3(x))\theta_j(x,y_3(x))+P_j(x,y_3(x))(\theta_j)_y(x,y_3(x))\} \nonumber \\ &=&
f_y(x,y_3(x),\theta(x,y_3(x)))+\sum_{j=1}^n P_j(x,y_3(x)) (\theta_j)_y(x,y_3(x))
\nonumber \\ &&- \sum_{j=1}^n \{(P_j)_y(x,y_3(x))\theta_j(x,y_3(x))+P_j(x,y_3(x))(\theta_j)_y(x,y_3(x))\}
\nonumber \\ &=&f_y(x,y_3(x),\theta(x,y_3(x)))- \sum_{j=1}^n (P_j)_y(x,y_3(x))\theta_j(x,y_3(x)). \end{eqnarray}

Thus, $$h_y(x,y)=f_y(x,y,\theta(x,y))- \sum_{j=1}^n (P_j)_y(x,y)\theta_j(x,y),\; \forall (x,y) \in D_1.$$

On the other hand, since $y_3(x)$ is stationary, we obtain

\begin{eqnarray}
0&=& f_y(x,y_3(x),\nabla y_3(x))-\sum_{j=1}^n\frac{d}{dx_j} f_{z_j}(x,y_3(x),\nabla y_3(x)) \nonumber \\ &=&
h_y(x,y_3(x))+ \sum_{j=1}^n(P_j)_y(x,y)\theta_j(x,y)\nonumber \\ &&-\sum_{j=1}^n\left(\frac{\partial P_j(x,y_3(x))}{\partial x_j}+(P_j)_y(x,y_3(x))\frac{\partial y_3(x)}{\partial x_j}\right)
\nonumber \\ &=&h_y(x,y_3(x))+ \sum_{j=1}^n (P_j)_y(x,y_3(x))\theta_j(x,y_3(x)) \nonumber \\ &&-\sum_{j=1}^n\left(\frac{\partial P_j(x,y_3(x))}{\partial x_j}+(P_j)_y(x,y_3(x))\theta_j(x,y_3(x))\right) \nonumber \\ &=&
 h_y(x,y_3(x))-\sum_{j=1}^n\left(\frac{\partial P_j(x,y_3(x))}{\partial x_j}\right).
 \end{eqnarray}

 Therefore, $$h_y(x,y)=\sum_{j=1}^n\left(\frac{\partial P_j(x,y)}{\partial x_j}\right), \forall (x,y) \in D_1$$

 Let $H_j(x,y)$ be such that $$\frac{\partial H_j(x,y)}{\partial y}=P_j(x,y).$$

From these two last lines, we get

$$h_y(x,y)=\sum_{j=1}^n\left(\frac{\partial (H_j)_y(x,y)}{\partial x_j}\right)=\left(\sum_{j=1}^n\frac{\partial H_j(x,y)}{\partial x_j}\right)_y,\; \forall (x,y) \in D_1$$
so that
$$h(x,y)=\sum_{j=1}^n\frac{\partial H_j(x,y)}{\partial x_k}+W(x),$$
for some $W:B_0 \rightarrow \mathbb{R}.$

Hence, assuming $D_1$ contains an open which contains $\overline{\mathcal{C}_0}$, where $\mathcal{C}_0= \{(x,y_0(x))\;:\; x \in \Omega\}$, for $y \in D$ sufficiently close to $y_0$ in $L^\infty$ norm, the generalized  Hilbert integral, denoted by $I(y)$, will be defined by
\begin{eqnarray}
I(y)&=& \int_\Omega h(x,y(x))\;dx+ \sum_{j=1}^n P_j(x,y(x))\frac{\partial y(x)}{\partial x_j}\;dx \nonumber \\ &=&
\int_\Omega \sum_{j=1}^n\left(\frac{\partial H_j(x,y(x))}{\partial x_j}+(H_j)_y(x,y(x))\frac{\partial y(x)}{\partial x_j}\right)\;dx
+\int_\Omega W(x)\;dx \nonumber \\ &=& \int_\Omega \sum_{j=1}^n \frac{d H_j(x,y(x))}{dx_j}\;dx+\int_\Omega W(x)\;dx
\nonumber \\ &=& \int_{\partial \Omega} \sum_{j=1}^n (-1)^{j+1} H_j(x,y(x))dx_1\wedge\cdots\wedge \widehat{dx_j} \wedge \cdots \wedge dx_n +\int_\Omega W(x)\;dx\nonumber \\
&=& W_1(y|_{\partial \Omega})\nonumber \\ &=& W_1((y_0)|_{\partial \Omega}),
\end{eqnarray}
so that such an integral is invariant, that is, it does not depend on $y.$

Finally, observe that

$$F(y)-F(y_0) =
\int_\Omega f(x,y(x),\nabla y(x))\;dx-\int_\Omega f(x,y_0(x),\nabla y_0(x))\;dx.$$

On the other hand,

\begin{eqnarray}
I(y) &=& \int_\Omega f(x,y(x),\theta(x,y))\;dx\nonumber \\ && +\sum_{j=1}^n\int_\Omega f_{z_j}(x,y(x),\theta(x,y(x))(\theta_j(x,y(x))-y_{x_j}(x))\;dx
\nonumber \\ &=&I(y_0) \nonumber \\ &=&\int_\Omega f(x,y_0(x),\theta(x,y_0(x)))\;dx \nonumber \\ &=& \int_\Omega f(x,y_0(x), \nabla y_0(x))\;dx \nonumber \\ &=& F(y_0).
\end{eqnarray}

Therefore,
\begin{eqnarray}
F(y)-F(y_0)&=& \int_\Omega [f(x,y(x),\nabla y(x)) -f(x,y(x),\theta(x,y))]\;dx
\nonumber \\ &&-\sum_{j=1}^n\int_\Omega f_{z_j}(x,y(x),\theta(x,y(x))(y_{x_j}(x)-\theta_j(x,y(x)))\;dx
\nonumber \\ &=& \int_\Omega \mathcal{E}(x,y(x), \theta(x,y(x)), \nabla y(x))\;dx,
\end{eqnarray}
 where\begin{eqnarray}\mathcal{E}(x,y(x),\theta(x,y(x)), \nabla y(x))&=&f(x,y(x),\nabla y(x)) \nonumber \\ &&-f(x,y(x),\theta(x,y))\nonumber \\ &&-\sum_{j=1}^nf_{z_j}(x,y(x),\theta(x,y(x))(y_{x_j}(x)-\theta_j(x,y(x))),\end{eqnarray}
is the Weierstrass Excess function.

With such results, we may prove the following theorem.

\begin{thm} Let $\Omega \subset \mathbb{R}^n$ be an open, bounded and simply connected set, with a regular (Lipschitzian) boundary denoted by
$\partial \Omega$. Let $V=C^1(\overline{\Omega})$ and let $f \in C^1(\overline{\Omega} \times \mathbb{R} \times \mathbb{R}^n)$ be such that $f(\underline{x},\underline{y},\mathbf{z})$ is convex. Let $F:D \rightarrow \mathbb{R}$ be defined by
$$F(y)=\int_\Omega f(x,y(x),\nabla y(x))\;dx,$$ where $$D=\{y \in V\;:\; y=y_1 \text{ em } \partial \Omega\}.$$

Let $y_0 \in D$ be a stationary function for $f$ which may be extended to $B_0$, being kept stationary in $B_0$, where $B_0$ has been specified above in this section.  Suppose we may define a field $\theta : D_1 \rightarrow \mathbb{R}^{n}$, also as it has been specified above in this section. Assume  $D_1 \subset \mathbb{R}^{n+1}$
 contains an open set which contains $\overline{\mathcal{C}_0}$, where

$$\mathcal{C}_0=\{(x,y_0(x))\;:\; x \in \Omega\}.$$

Under such hypotheses, there exists $\delta>0$ such that
$$F(y) \geq F(y_0),\; \forall y \in B_\delta(y_0)\cap D,$$

where $$B_{\delta}(y_0)=\{y \in V\;:\; \|y-y_0\|_\infty <\delta\}.$$
\end{thm}
\begin{proof}
From the hypotheses and from the exposed above in this section, there exists $\delta>0$ such that $\theta(x,y(x))$ is well defined for all  $y \in D$ such that $$\|y-y_0\|_\infty< \delta.$$

Let $y \in B_\delta(y_0)\cap D.$ Thus,
 since $f(\underline{x},\underline{y},\mathbf{z})$ is convex, we have  \begin{eqnarray}\mathcal{E}(x,y(x),\theta(x,y(x)), \nabla y(x))&=&f(x,y(x),\nabla y(x)) \nonumber \\ &&-f(x,y(x)\theta(x,y))\nonumber \\ &&-\sum_{j=1}^nf_{z_j}(x,y(x),\theta(x,y(x))(y_{x_j}(x)-\theta_j(x,y(x))) \nonumber \\ &\geq& 0, \text{ in } \Omega\end{eqnarray}
so that,
$$F(y)-F(y_0)=\int_\Omega \mathcal{E}(x,y(x),\theta(x,y(x)),\nabla y(x))\;dx  \geq 0, \; \forall y \in B_\delta(y_0)\cap D.$$
\end{proof}
The proof is complete.
\section{Central fields and the vectorial case in the calculus of variations}

Let $\Omega \subset \mathbb{R}^n$ be an open, bounded, simply connected set with a regular (Lipschitzian) boundary denoted by $\partial \Omega$.

Let $f \in C^1(\overline{\Omega} \times \mathbb{R}^N \times \mathbb{R}^{Nn})$ ($f(x,\mathbf{y},\mathbf{z})$) and $V=C^1(\overline{\Omega};\mathbb{R}^N).$
We suppose $f$ is convex in $\mathbf{z}$, which we denote by $f(\underline{x},\underline{\mathbf{y}},\mathbf{z})$  be convex.

  Choose $\tilde{x}_1,\ldots, \tilde{x}_n,\tilde{y}_1,\ldots,\tilde{y}_n \in \mathbb{R}$ such that if $x=(x_1,\ldots,x_n) \in \Omega$, then $$\tilde{x}_k<x_k<\tilde{y}_k \; \forall k \in \{1, \ldots, n\}$$

Suppose that $\mathbf{y}_0 \in V_1=C^1(B_0;\mathbb{R}^N)$ is stationary for
$f$ in $B_0=\prod_{k=1}^n [\tilde{x}_k, \tilde{y}_k]$, that is, suppose that $$\sum_{k=1}^n\frac{d}{dx_k} f_{z_{jk}}(x,\mathbf{y}_0(x),\nabla \mathbf{y}_0(x))=f_{y_j}(x,\mathbf{y}_0(x),\nabla \mathbf{y}_0(x)),\; \forall x \in B_0
, \; \forall j \in \{1, \ldots,N\}.$$
We also assume $f$ to be of $C^1$ class in $B_0 \times \mathbb{R}^N \times \mathbb{R}^{Nn}$.

Consider the problem of minimizing $F:D \rightarrow \mathbb{R}$ where
$$F(\mathbf{y})=\int_\Omega f(x,\mathbf{y}(x), \nabla \mathbf{y}(x))\;dx,$$ and where
$$D=\{ \mathbf{y} \in V\;:\; \mathbf{y}=\mathbf{y}_0 \text{ in } \partial \Omega\}.$$

Assume there exists a family of stationary functions $\mathcal{F}$, such that for each $(x,\mathbf{y}) \in D_1 \subset \mathbb{R}^{n+N}$,
there exists a unique stationary function $\mathbf{y}_3 \in V_1$ for $f$ in $\mathcal{F}$, such that
$$(x,\mathbf{y})=(x,\mathbf{y}_3(x)).$$

More specifically, we define $\mathcal{F}$ as a subset of $\mathcal{F}_1$, where
$$\mathcal{F}_1=\{\Phi(t,\Lambda(x,\mathbf{y})) \text{ stationary for } f \text{ such that } \Phi(x,\Lambda(x,\mathbf{y}))=\mathbf{y} \text{ for a }
\Lambda \in B_r(\mathbf{0}) \subset \mathbb{R}^N\}$$ such that $B_r(\mathbf{0})$ is an open ball of center $\mathbf{0}$ and radius $r$, for some $r>0$.

Here, $\Phi$ is stationary        and  $$\Phi((t_1,\ldots,t_k=\tilde{x}_k,\ldots,t_n),\Lambda(x,\mathbf{y}))=\mathbf{y}_0(t_1,\ldots,t_k=\tilde{x}_k,\ldots,t_n), \text{ on } \partial B_0$$
and $$\nabla (\Phi)_j((t_1,\ldots,t_k=\tilde{x}_k,\ldots,t_n),\Lambda(x,\mathbf{y}))\cdot \mathbf{n}=\Lambda_j, \text{ on } \partial B_0,\; \forall k \in \{1,\ldots,n\},\;
 j \in \{1,\ldots,N\}.$$
Here, $\mathbf{n}$ denotes the outward normal field to $\partial B_0.$

Define the field $\theta :D_1 \subset \mathbb{R}^{n+N} \rightarrow \mathbb{R}^{Nn}$ by $$\theta_j(x,\mathbf{y})=\nabla (y_3)_j(x),$$ where
as above indicated, $\mathbf{y}_3$ is such that
$$(x,\mathbf{y})=(x,\mathbf{y}_3(x)),$$ and thus $$\theta_j(x,(\mathbf{y}_3)(x))=\nabla (y_3)_j(x), \; \forall x \in B_0.$$

At this point we assume the hypotheses of the implicit function theorem so that $\Phi_j(x,\Lambda(x,\mathbf{y}))$ is of $C^1$ class and $\theta_j(x,\mathbf{y})$ is continuous (in fact, since $\Phi_j$ is stationary, the partial derivatives of $\theta_j(x,\mathbf{y})$ are well defined, \; $\forall j \in \{1,\ldots,N\}$).

Define also $$h(x,\mathbf{y})=f(x,\mathbf{y},\theta(x,\mathbf{y}))-\sum_{j=1}^n\sum_{k=1}^nf_{z_{jk}}(x,y,\theta(x,\mathbf{y})) \theta_{jk}(x,\mathbf{y})$$
and $$P_{jk}(x,\mathbf{y})=f_{z_{jk}}(x,\mathbf{y},\theta(x,\mathbf{y})),\; \forall j \in \{1, \ldots,N\},\; k \in \{1, \ldots,n\}.$$

Observe that \begin{eqnarray}h_{y_j}(x,\mathbf{y}_3(x))&=&f_{y_j}(x,\mathbf{y}_3(x),\theta(x,\mathbf{y}_3(x)))+\sum_{l=1}^N\sum_{k=1}^n f_{z_{lk}}(x,\mathbf{y}_3(x),\theta(x,\mathbf{y}_3(x))) (\theta_{lk})_{y_j}(x,\mathbf{y}_3(x))
\nonumber \\ &&- \sum_{l=1}^N \sum_{k=1}^n \{(P_{lk})_{y_j}(x,\mathbf{y}_3(x))\theta_{lk}(x,\mathbf{y}_3(x))+P_{lk}(x,\mathbf{y}_3(x))(\theta_{lk})_{y_j}
(x,\mathbf{y}_3(x))\} \nonumber \\ &=&
f_{y_j}(x,\mathbf{y}_3(x),\theta(x,\mathbf{y}_3(x)))+\sum_{l=1}^N\sum_{k=1}^n P_{lk}(x,\mathbf{y}_3(x)) (\theta_{lk})_{y_j}(x,\mathbf{y}_3(x))
\nonumber \\ &&- \sum_{j=1}^n \{(P_j)_y(x,\mathbf{y}_3(x))\theta_j(x,\mathbf{y}_3(x))+P_j(x,\mathbf{y}_3(x))(\theta_j)_y(x,\mathbf{y}_3(x))\}
\nonumber \\ &=&f_{y_j}(x,\mathbf{y}_3(x),\theta(x,\mathbf{y}_3(x)))- \sum_{l=1}^N \sum_{k=1}^n (P_{lk})_{y_j}(x,\mathbf{y}_3(x))\theta_{lk}(x,\mathbf{y}_3(x)). \end{eqnarray}

Thus, $$h_{y_j}(x,\mathbf{y})=f_{y_j}(x,\mathbf{y},\theta(x,\mathbf{y}))- \sum_{l=1}^N\sum_{k=1}^n (P_{lk})_{y_j}(x,\mathbf{y})\theta_{lk}(x,\mathbf{y}),\;
\forall (x,\mathbf{y}) \in D_1.$$

On the other hand, considering that $\mathbf{y}_3(x)$ is stationary, we obtain $H_k(x,\mathbf{y})$ such that $$P_{jk}(x, \mathbf{y})=\frac{\partial H_k(x,\mathbf{y})}{\partial y_j},\; \forall j \in \{1, \ldots,N\},\; k \in \{1, \ldots,n\}.$$

Indeed, we define $$H_k(x,\mathbf{y})=sta_{\phi \in \hat{D}_k}\int_{\tilde{x}_k}^{x_k}f(x_1,\ldots,t_k, \ldots, x_n,\phi_1(t_k),\ldots,\phi_N(t_k),
\nabla \tilde{\phi}(x,t_k))\;dt_k,$$
 where, denoting $\tilde{t}_k=(x_1, \ldots, t_k, \ldots,x_n)$, we have that
\begin{equation}\nabla \tilde{\phi}(x,t_k)=\left[
\begin{array}{cccccc}
 ((y_3)_1)_{x_1}(\tilde{t}_k) &  ((y_3)_1)_{x_2}(\tilde{t}_k)& \cdots & \frac{\partial \phi_1(t_k)}{\partial t_k}& \cdots &   ((y_3)_1)_{x_n}(\tilde{t}_k)
 \\
 ((y_3)_2)_{x_1}(\tilde{t}_k) &  ((y_3)_2)_{x_2}(\tilde{t}_k)& \cdots & \frac{\partial \phi_2(t_k)}{\partial t_k}& \cdots &   ((y_3)_2)_{x_n}(\tilde{t}_k)
 \\
 \vdots & \vdots &  \cdots & \vdots & \ddots & \vdots
 \\
 ((y_3)_N)_{x_1}(\tilde{t}_k) &  ((y_3)_N)_{x_2}(\tilde{t}_k)& \cdots & \frac{\partial \phi_N(t_k)}{\partial t_k}& \cdots &   ((y_3)_N)_{x_n}(\tilde{t}_k)
  \end{array} \right]_{N \times n},\end{equation}
and where
$$\hat{D}_k=\{\phi \in C^1([\tilde{x}_k,x_k];\mathbb{R}^N)\;:\; \phi(\tilde{x}_k)=\mathbf{y}_3(\tilde{x}_k) \text{ and } \phi(x_k)=\mathbf{y}_3(x_k)\},$$
$\forall k \in \{1,\ldots,n\}.$

Observe that, since $\phi(t_k)$ is stationary, we have

\begin{equation}\label{uk001}f_{y_j}[\phi(t_k)]-\frac{d }{dt_k} f_{z_{jk}}[\phi(t_k)]=0, \text{ in } [\tilde{x}_k,x_k], \; \forall j \in \{1,\ldots,N\}
\end{equation}
where generically, we denote $$f[\phi(t_k)]=f(x,\phi(t_k),\nabla \tilde{\phi}(t_k)).$$

Observe that from these Euler-Lagrange equations we may obviously obtain $$\phi(t_k)=\mathbf{y}_3(x_1,\ldots, t_k, \ldots, x_n).$$

At this point we shall also denote
$$\phi_j(t_k)\equiv \phi_j(t_k,\tilde{\Lambda}(x,\mathbf{y})),$$ where
$$\phi_j(\tilde{x}_k,\tilde{\Lambda}(x,\mathbf{y}))=(y_3)_j(x_1,\ldots,\tilde{x}_k, \ldots, x_n),$$

$$\frac{\partial \phi_j(\tilde{x}_k, \tilde{\Lambda}(x ,\mathbf{y}))}{\partial t_k}=\tilde{\Lambda}_j,$$
and where $\tilde{\Lambda}(x,\mathbf{y})$ is such that
$$\phi_j(x_k,\tilde{\Lambda}(x,\mathbf{y}))=(y_3)_j(x_1,\ldots,x_k, \ldots, x_n).$$

Therefore,

$$H_k(x,\mathbf{y})=\int_{\tilde{x}_k}^{x_k} f[\phi(t_k, \tilde{\Lambda}(x,\mathbf{y})]\;dt_k,$$
and thus,

\begin{eqnarray}[H_k(x,\mathbf{y})]_{y_j}&=&\int_{\tilde{x}_k}^{x_k} f_{y_l}[\phi(t_k, \tilde{\Lambda}(x,\mathbf{y})](\phi_l)_{\tilde{\Lambda}} \tilde{\Lambda}_{y_j}(x,\mathbf{y})\;dt_k \nonumber \\ &&+\int_{\tilde{x}_k}^{x_k} f_{z_{lk}}[\phi(t_k, \tilde{\Lambda}(x,\mathbf{y})](\phi_l)_{t_k \tilde{\Lambda}} \tilde{\Lambda}_{y_j}(x,\mathbf{y})\;dt_k \nonumber \\ &=&\int_{\tilde{x}_k}^{x_k} f_{y_l}[\phi(t_k, \tilde{\Lambda}(x,\mathbf{y})](\phi_l)_{\tilde{\Lambda}} \tilde{\Lambda}_{y_j}(x,\mathbf{y})\;dt_k \nonumber \\ &&+\int_{\tilde{x}_k}^{x_k} f_{z_{lk}}[\phi(t_k, \tilde{\Lambda}(x,\mathbf{y})][(\phi_l)_{ \tilde{\Lambda}}]_{t_k} \tilde{\Lambda}_{y_j}(x,\mathbf{y})\;dt_k. \end{eqnarray}

From this and (\ref{uk001}), we obtain \begin{eqnarray}\label{us001}[H_k(x,\mathbf{y})]_{y_j}&=&\int_{\tilde{x}_k}^{x_k} \frac{d}{dt_k}f_{z_{lk}}[\phi(t_k, \tilde{\Lambda}(x,\mathbf{y})](\phi_l)_{\tilde{\Lambda}} \tilde{\Lambda}_{y_j}(x,\mathbf{y})\;dt_k \nonumber \\ &&+\int_{\tilde{x}_k}^{x_k} f_{z_{lk}}[\phi(t_k, \tilde{\Lambda}(x,\mathbf{y})][(\phi_l)_{\tilde{\Lambda}}]_{t_k} \tilde{\Lambda}_{y_j}(x,\mathbf{y})\;dt_k \nonumber \\ &=&\int_{\tilde{x}_k}^{x_k} \frac{d}{d t_k}\{f_{z_{lk}}[\phi(t_k, \tilde{\Lambda}(x,\mathbf{y})](\phi_l)_{ \tilde{\Lambda}}]\}\;dt_k \tilde{\Lambda}_{y_j}(x,\mathbf{y}) \nonumber \\ &=& \{f_{z_{lk}}[\phi(t_k, \tilde{\Lambda}(x,\mathbf{y})](\phi_l)_{ \tilde{\Lambda}}]\}|_{t_k=\tilde{x_k}}^{t_k=x_k}\tilde{\Lambda}_{y_j}(x,\mathbf{y}) \end{eqnarray}

On the other hand,

$$\phi_l(\tilde{x}_k, \tilde{\Lambda}(x,\mathbf{y}))=(y_3)_l(x_1, \ldots,\tilde{x}_k, \ldots, x_n)=(y_0)_l(x_1, \ldots, \tilde{x}_k,\ldots, x_n),$$
which does not depend on  $\tilde{\Lambda}$, so that $$(\phi_l)_{\Lambda}(\tilde{x}_k,\tilde{\Lambda}(x,\mathbf{y}))= \mathbf{0}.$$

Also, $$\phi_l(x_k, \tilde{\Lambda}(x,\mathbf{y}))=y_l$$ and thus
$$[\phi_l(x_k, \tilde{\Lambda}(x,\mathbf{y}))]_{y_j}=\frac{\partial y_l}{\partial y_j}=\delta_{lj}.$$

Hence,
$$(\phi_l)_{\tilde{\Lambda}}(x_k,\tilde{\Lambda}(x,\mathbf{y}))\tilde{\Lambda}_{y_j}= \delta_{lj}.$$

From these last results and from (\ref{us001}), we obtain,

\begin{eqnarray}[H_k(x,\mathbf{y})]_{y_j}&=&f_{z_{lk}}[\phi(x_k,\tilde{\Lambda}(x,\mathbf{y}))] \delta_{lj} \nonumber \\
&=& f_{z_{jk}}[\phi(x_k,\tilde{\Lambda}(x,\mathbf{y}))] \nonumber \\ &=& f_{z_{jk}}(x,\mathbf{y}, \theta(x,\mathbf{y}))
\nonumber \\ &=& P_{jk}(x,\mathbf{y}),\; \forall j \in \{1,\ldots,N\}, \; k \in \{1,\ldots, n\}.
\end{eqnarray}

Therefore,

\begin{eqnarray}
0&=& f_{y_j}(x,\mathbf{y}_3(x),\nabla \mathbf{y}_3(x))-\sum_{k=1}^n\frac{d}{dx_k} f_{z_{jk}}(x,\mathbf{y}_3(x),\nabla \mathbf{y}_3(x)) \nonumber \\ &=&
h_{y_j}(x,\mathbf{y}_3(x))+ \sum_{l=1}^N\sum_{k=1}^n(P_{lk})_{y_j}(x,\mathbf{y}_3(x))\theta_{lk}(x,\mathbf{y}_3(x))\nonumber \\ &&-\sum_{k=1}^n\left(\frac{\partial P_{jk}(x,\mathbf{y}_3(x))}{\partial x_k}+\sum_{l=1}^N(P_{jk})_{y_l}(x,\mathbf{y}_3(x))\frac{\partial (y_3)_l(x)}{\partial x_k}\right)
\nonumber \\ &=&h_{y_j}(x,\mathbf{y}_3(x))+ \sum_{l=1}^N\sum_{k=1}^n (P_{lk})_{y_j}(x,\mathbf{y}_3(x))\theta_{lk}(x,\mathbf{y}_3(x)) \nonumber \\ &&-\sum_{k=1}^n\left(\frac{\partial P_{jk}(x,\mathbf{y}_3(x))}{\partial x_k}+\sum_{l=1}^N(P_{jk})_{y_l}(x,\mathbf{y}_3(x))\theta_{lk}(x,\mathbf{y}_3(x))\right) \nonumber \\ &=&
 h_{y_j}(x,\mathbf{y}_3(x))-\sum_{k=1}^n\left(\frac{\partial (H_k)_{y_j}(x,\mathbf{y}_3(x))}{\partial x_k}\right).
 \end{eqnarray}

 Thus, $$h_{y_j}(x,\mathbf{y})=\sum_{k=1}^n\left(\frac{\partial H_k(x,\mathbf{y})}{\partial x_k}\right)_{y_j}, \forall (x,\mathbf{y}) \in D_1$$
so that
$$h(x,\mathbf{y})=\sum_{k=1}^n\frac{\partial H_k(x,\mathbf{y})}{\partial x_k}+W(x),$$
for some $W:B_0 \rightarrow \mathbb{R}.$

Hence, assuming $D_1$ contains an open set which contains $\overline{\mathcal{C}_0}$, where $$\mathcal{C}_0=\{(x,\mathbf{y}_0(x))\;:\; x \in \Omega\},$$ for $\mathbf{y} \in D$ sufficiently close to $\mathbf{y}_0$ in $L^\infty$ norm, the generalized Hilbert integral, denoted by $I(\mathbf{y})$, will be defined as
\begin{eqnarray}
I(\mathbf{y})&=& \int_\Omega h(x,\mathbf{y}(x))\;dx+ \sum_{j=1}^N \sum_{k=1}^n P_{jk}(x,\mathbf{y}(x))\frac{\partial y_j(x)}{\partial x_k}\;dx \nonumber \\ &=&
\int_\Omega \sum_{k=1}^n\left(\frac{\partial H_k(x,\mathbf{y}(x))}{\partial x_k}+\sum_{j=1}^n(H_k)_{y_j}(x,\mathbf{y}(x))\frac{\partial y_j(x)}{\partial x_k}\right)\;dx
+\int_\Omega W(x)\;dx \nonumber \\ &=& \int_\Omega \sum_{k=1}^n \frac{d H_k(x,\mathbf{y}(x))}{dx_k}\;dx+\int_\Omega W(x)\;dx
\nonumber \\ &=& \int_{\partial \Omega} \sum_{k=1}^n (-1)^{k+1} H_j(x,\mathbf{y}(x))dx_1\wedge\cdots\wedge \widehat{dx_j} \wedge \cdots \wedge dx_n +\int_\Omega W(x)\;dx\nonumber \\
&=& W_1(\mathbf{y}|_{\partial \Omega})\nonumber \\ &=& W_1((\mathbf{y}_1)|_{\partial \Omega}),
\end{eqnarray}
so that such an integral is invariant, that is, it does not depend on $\mathbf{y}.$

Finally, observe that

$$F(\mathbf{y})-F(\mathbf{y}_0) =
\int_\Omega f(x,\mathbf{y}(x),\nabla \mathbf{y}(x))\;dx-\int_\Omega f(x,\mathbf{y}_0(x),\nabla \mathbf{y}_0(x))\;dx.$$

On the other hand,

\begin{eqnarray}
I(\mathbf{y}) &=& \int_\Omega f(x,\mathbf{y}(x),\theta(x,\mathbf{y}(x)))\;dx\nonumber \\ && +\sum_{j=1}^N\sum_{k=1}^n\int_\Omega f_{z_{jk}}(x,\mathbf{y}(x),\theta(x,\mathbf{y}(x))(\theta_{jk}(x,\mathbf{y}(x))-(y_j)_{x_k}(x))\;dx
\nonumber \\ &=&I(\mathbf{y}_0) \nonumber \\ &=&\int_\Omega f(x,\mathbf{y}_0(x),\theta(x,\mathbf{y}_0(x)))\;dx
\nonumber \\ &=& \int_\Omega f(x,\mathbf{y}_0(x),\nabla \mathbf{y}_0(x))\;dx\nonumber \\ &=& F(\mathbf{y}_0).
\end{eqnarray}

Thus,
\begin{eqnarray}
F(\mathbf{y})-F(\mathbf{y}_0)&=& \int_\Omega [f(x,\mathbf{y}(x),\nabla \mathbf{y}(x)) -f(x,\mathbf{y}(x),\theta(x,\mathbf{y}(x)))]\;dx
\nonumber \\ &&-\sum_{j=1}^N \sum_{k=1}^n\int_\Omega f_{z_{jk}}(x,\mathbf{y}(x),\theta(x,\mathbf{y}(x))((y_j)_{x_k}(x)-\theta_{jk}(x,\mathbf{y}(x)))\;dx
\nonumber \\ &=& \int_\Omega \mathcal{E}(x,\mathbf{y}(x), \theta(x,\mathbf{y}(x)), \nabla \mathbf{y}(x))\;dx,
\end{eqnarray}
 where
 \begin{eqnarray}\mathcal{E}(x,\mathbf{y}(x),\theta(x,\mathbf{y}(x)), \nabla \mathbf{y}(x))&=&f(x,\mathbf{y}(x),\nabla \mathbf{y}(x)) \nonumber \\ &&-f(x,\mathbf{y}(x),\theta(x,\mathbf{y}(x)))\nonumber \\ &&-\sum_{j=1}^N \sum_{k=1}^nf_{z_{jk}}(x,\mathbf{y}(x),\theta(x,\mathbf{y}(x))((y_j)_{x_k}(x)-\theta_{jk}(x,\mathbf{y}(x)))\nonumber\end{eqnarray}
is the Weierstrass Excess function.

With such results, we may prove the following result.

\begin{thm} Let $\Omega \subset \mathbb{R}^n$ be an open, bounded, simply connected set with a regular (Lipschitzian) boundary denoted by
$\partial \Omega$. Let $V=C^1(\overline{\Omega};\mathbb{R}^N)$ and let $f \in C^1(\overline{\Omega} \times \mathbb{R}^N \times \mathbb{R}^{Nn})$ be such that
$f(\underline{x},\underline{\mathbf{y}},\mathbf{z})$ is convex. Let $F:D \rightarrow \mathbb{R}$ be defined by
$$F(\mathbf{y})=\int_\Omega f(x,\mathbf{y}(x),\nabla \mathbf{y}(x))\;dx,$$ where $$D=\{\mathbf{y} \in V\;:\; \mathbf{y}=\mathbf{y}_1 \text{ on } \partial \Omega\}.$$

Let $\mathbf{y}_0 \in D$ be an stationary function for $f$ which may be extended to $B_0$, keeping it stationary in $B_0$, where $B_0$ has been
specified above in this section. Suppose we may define a field $\theta : D_1 \rightarrow \mathbb{R}^{Nn}$, also as specified above in this section. Assume  $D_1 \subset \mathbb{R}^{n+N}$
 contains an open set which contains $\overline{\mathcal{C}_0}$, where

$$\mathcal{C}_0=\{(x,\mathbf{y}_0(x))\;:\; x \in \Omega\}.$$

Under such hypotheses, there exists $\delta>0$ such that
$$F(\mathbf{y}) \geq F(\mathbf{y}_0),\; \forall \mathbf{y} \in B_\delta(\mathbf{y}_0)\cap D,$$

where $$B_{\delta}(\mathbf{y}_0)=\{\mathbf{y} \in V\;:\; \|\mathbf{y}-\mathbf{y}_0\|_\infty <\delta\}.$$
\end{thm}
\begin{proof}
From the hypotheses and from the exposed above in this section, there exists $\delta>0$ such that $\theta(x,\mathbf{y}(x))$ is well defined for each $\mathbf{y} \in D$ such that $$\|\mathbf{y}-\mathbf{y}_0\|_\infty< \delta.$$

Let $\mathbf{y} \in B_\delta(\mathbf{y}_0)\cap D.$ Thus,
 since  $f(\underline{x},\underline{\mathbf{y}},\mathbf{z})$ is convex, we have
 \begin{eqnarray}\mathcal{E}(x,\mathbf{y}(x),\theta(x,\mathbf{y}(x)), \nabla \mathbf{y}(x))&=&f(x,\mathbf{y}(x),\nabla \mathbf{y}(x)) \nonumber \\ &&-f(x,\mathbf{y}(x),\theta(x,\mathbf{y}(x)))\nonumber \\ &&-\sum_{j=1}^N \sum_{k=1}^nf_{z_{jk}}(x,\mathbf{y}(x),\theta(x,\mathbf{y}(x))((y_j)_{x_j}(x)
 -\theta_{jk}(x,\mathbf{y}(x))) \nonumber \\ &\geq& 0, \text{ in } \Omega\end{eqnarray}
so that,
$$F(\mathbf{y})-F(\mathbf{y}_0)=\int_\Omega \mathcal{E}(x,\mathbf{y}(x),\theta(x,\mathbf{y}(x)),\nabla \mathbf{y}(x))\;dx  \geq 0, \; \forall \mathbf{y} \in B_\delta(\mathbf{y}_0)\cap D.$$

The proof is complete.
\end{proof}

\end{document}